\documentclass[opre,nonblindrev]{informs4}
\usepackage{eqndefns-left}
\Equationvalidatefalse
\RequirePackage{tgtermes}
\RequirePackage{newtxtext}
\RequirePackage{newtxmath}
\RequirePackage{bm}
\RequirePackage{endnotes}

\OneAndAHalfSpacedXI

\usepackage{natbib}
 \bibpunct[, ]{(}{)}{,}{a}{}{,}%
 \def\bibfont{\small}%
 \def\BIBand{and}%
\usepackage{enumitem}
\usepackage{color}
\usepackage{array}
\usepackage{multirow}
\usepackage{verbatim}
\usepackage{comment}
\usepackage{graphicx}
\graphicspath{{./}}
\usepackage{algorithm}
\usepackage{algorithmic}
\usepackage{subcaption,caption}
\usepackage{adjustbox}
\usepackage{booktabs}
\usepackage{placeins}
\usepackage{array}
\usepackage{makecell}
\usepackage[hypertexnames=false,hyperindex=true,pdfpagemode=UseOutlines,bookmarksnumbered=true,bookmarksopen=true,bookmarksopenlevel=2,pdfstartview=FitH,hidelinks]{hyperref}

\TheoremsNumberedThrough     
\ECRepeatTheorems

\EquationsNumberedThrough    

\MANUSCRIPTNO{}

\makeatletter
\def\theARTICLETOP{}
\def\theARTICLEABSTRACT{%
  \par\nobreak\vskip12pt\noindent%
  \begin{minipage}{\textwidth}\parindent1em\ABSfont
  \noindent\theABSTRACT\par\vskip5pt
  \theKEYWORDS
  \end{minipage}\par\vspace*{12pt}}
\makeatother
\LRHSecondLine{}
\RRHSecondLine{}
\hypersetup{pdftitle={Information Limits of Multistage Inventory Control: Learning, Valuation, and Censoring},pdfauthor={Hanzhang Qin, David Simchi-Levi, Ruihao Zhu}}

\begin{document}


\RUNAUTHOR{Qin, Simchi-Levi, and Zhu}
\ECRUNAUTHOR{Qin, Simchi-Levi, and Zhu}

\RUNTITLE{Information Limits of Inventory Control}

\TITLE{Information Limits of Multistage Inventory Control: Learning, Valuation, and Censoring}


\ARTICLEAUTHORS{
\AUTHOR{Hanzhang Qin}
\AFF{Department of Industrial Systems Engineering and Management and Institute of Operations Research and Analytics, National University of Singapore, Singapore 117576, \EMAIL{hzqin@nus.edu.sg}}
 \AUTHOR{David Simchi-Levi}
 \AFF{Institute for Data, Systems, and Society, Department of Civil and Environmental Engineering, and Operations Research Center, Massachusetts Institute of Technology, Cambridge, MA 02139, \EMAIL{dslevi@mit.edu}}
 \AUTHOR{Ruihao Zhu}
 \AFF{Cornell SC Johnson College of Business, Ithaca, NY 14853,
 	\EMAIL{ruihao.zhu@cornell.edu}} 
}

\ABSTRACT{
Learning how to replenish inventory and estimating the cost of doing so
require different information. We establish information-theoretic lower bounds
for offline multistage inventory control with independent bounded demand.
For fixed additive error in total expected cost, policy learning requires
cubically many scalar observations in the horizon in the worst case, even
under full observation and adaptive sample allocation. This matches the
known empirical-risk upper order. Stationarity permits a
quadratic policy-learning guarantee, yet valuation with inherited stock can
remain cubically difficult even when the optimal policy is known. Censoring
sharpens this distinction: optimal decisions can be learnable while absolute
costs remain unidentified. We give a carry-safe coverage condition under
which truncation preserves inventory transitions and policy gaps exactly.
Combining this reduction with a matching lower bound yields the sharp
censored-learning rate, with raw sample requirements inversely proportional
to the usable fraction of sales logs. The same qualified observations can
check coverage and fit the policy under a joint error guarantee. The results
isolate information limits for decisions and valuation under an explicit
conditionally independent logging protocol.
}
\maketitle

\section{Introduction}
Inventory decisions and their financial evaluation rely on the same demand
data, but need not require the same information. Stockout-censored sales may
identify an optimal replenishment policy without revealing its shortage costs.
We study this distinction in offline multistage inventory control with
independent bounded demands, linear holding and shortage costs, and zero
ordering cost and lead time. Known demand laws admit optimal period-specific
base-stock levels \citep{porteus2002foundations,levi2007provably}.

Our main result is an \emph{algorithm-independent cubic scalar-sample lower
bound} for additive-error policy learning in nonstationary finite-horizon
inventory control. It holds under full observation in the backlog model,
allows adaptive sample allocation, and covers arbitrary feasible randomized
policies. Together with the empirical-risk upper bound, it identifies the
worst-case data requirement of the inventory problem itself.

Stationarity, a.k.a. IID demand and time-homogeneous costs, permits a \emph{quadratic
policy-learning guarantee} uniformly over initial inventory, but valuation
can still require cubically many samples. In the hard instance, the optimal
rule is known; a rare demand determines how long inherited stock accumulates
holding costs. Starting empty restores a quadratic valuation guarantee.

For censored lost-sales data, we give carry-safe caps that contain optimal
replenishment levels and carried inventory. Truncating demand at these caps
preserves transitions and all cap-feasible policy gaps exactly. Product empirical risk
minimization attains the \emph{matching learning limit} from logs reaching the caps,
with raw sample requirements inversely proportional to their usable fraction.
The same observations can check coverage and fit the policy under our
conditionally independent logging protocol, while absolute costs may remain
unidentified.

\paragraph{Relation to prior work.}
Inventory-specific sample guarantees
\citep{levi2007provably,cheung2019sampling,qin2019data} and multistage sample
average approximation \citep{shapiro2005complexity} provide the broader
context. Single-period results address relative regret and precise sample
requirements \citep{levi2015data,besbes2021big}; \citet{huang2025lower}
survey information-theoretic lower bounds, including censored newsvendor
models. \citet{zhang2025series} give sample bounds for serial inventory
systems with a matching relative-accuracy lower order. We instead establish
horizon dependence for additive total-cost error with nonstationary demand,
adaptive sample allocation, and carryover. For attainability, we use
\citet{xie2026vc}: Corollary~4 controls horizon-normalized base-stock losses,
and Section~4.5 links Product ERM to empirical dynamic programming.

Censored-inventory learning includes adaptive policies
\citep{huh2009nonparametric,huh2011adaptive} and offline estimation of dynamic
$(s,S)$ policies with asymptotic confidence intervals
\citep{ban2019confidence}. These precedents distinguish learning while
controlling data collection from learning from a prescribed archive.
\citet{bu2020offline} introduce an \emph{observable boundary} measuring the
recoverable range of base demand in offline pricing, and characterize when censored data identify an optimal decision.
\citet{hssaine2024censored} characterize censored-newsvendor identification
and minimax regret; \citet{kumar2026value} give exact worst-case regret
analyses of newsvendor policies under censoring. \citet{fan2025censored}
study sample complexity in single-period and infinite-horizon inventory
models. Our contribution is the finite-horizon carryover structure:
\emph{carry-safe coverage supports an exact dynamic reduction and a matching
policy-learning lower bound}. 

Section~\ref{section: inventory preliminaries} defines the model and tasks.
Section~\ref{section: LB} gives lower bounds. Sections~\ref{section: VI}
and~\ref{section: censored} establish full-observation and censored
attainability. All proofs are in the main text, and numerical results are in the electronic companion.

\section{Model and Observation Protocols}
\label{section: inventory preliminaries}
For a positive integer \(n\), write \([n]=\{1,\ldots,n\}\), and put
\(u^+=\max\{u,0\}\). All logarithms are natural; a tilde on an asymptotic
order suppresses logarithmic factors. We use
\(\|f\|_\infty=\sup_u|f(u)|\) on the stated domain.

\subsection{Inventory Decisions and Values}
The horizon is \(T\) periods. Independent demands \(D_t\) have laws
\(P_t\) supported on \([0,\bar D]\), where \(\bar D>0\).
Holding and shortage coefficients \(h_t,p_t>0\) are known, and
\(C_c=\max_{t\in[T]}\max\{h_t,p_t\}\). We maintain zero fixed and
per-unit ordering costs, zero lead time, and zero terminal salvage value.
Starting period \(t\) with inventory \(x_t\), the retailer chooses a
post-order level \(y_t\geq\max\{x_t,0\}\) before observing demand.
The realized and expected one-period costs are
\[
c_t(y,d)=h_t(y-d)^++p_t(d-y)^+,
\qquad C_t(y)=\mathbb E[c_t(y,D_t)].
\]
The next state is \(s^{\rm B}(y,d)=y-d\) under backlog and
\(s^{\rm L}(y,d)=(y-d)^+\) under lost sales. Demand is observed in the
former case; only sales \(\min\{D_t,y_t\}\) are observed in the latter.
Under these primitives, backlog is cleared freely at the next decision:
the cost equivalence below is specific to zero ordering cost and lead time,
not a claim for general backlogging models.

A feasible policy \(\pi\) chooses each action using the information then available;
randomization is allowed. Write \(V_t^\pi(x)\) for its expected total cost
from period \(t\) and state \(x\), and \(V_t^*(x)=\inf_\pi V_t^\pi(x)\).
For either transition \(s\in\{s^{\rm B},s^{\rm L}\}\), the Bellman equations are
\[
V_{T+1}^*=0,\qquad
U_t(y)=C_t(y)+\mathbb E[V_{t+1}^*(s(y,D_t))],\qquad
V_t^*(x)=\min_{y\geq\max\{x,0\}}U_t(y).
\]
Let \(S_t^*=\min\arg\min_{y\geq0}U_t(y)\) be the smallest optimal
base-stock level and \(S_t^m=\max\arg\min_{y\in[0,\bar D]}C_t(y)\)
the largest myopic minimizer. A nonnegative base-stock vector
\(S=(S_1,\ldots,S_T)\) specifies the Markov policy
\(\pi_t^S(x)=\max\{x,S_t\}\); write \(S^*=(S_1^*,\ldots,S_T^*)\).

\begin{lemma}[Myopic upper bound]
\label{lemma:myopic base stock}
In both models, \(\pi^{S^*}\) is optimal and
\(0\leq S_t^*\leq S_t^m\leq\bar D\) for every \(t\in[T]\).
\end{lemma}
\begin{proof}{Proof.}
Free replenishment implies \(V_t^*(x)=V_t^*(x^+)\) for negative backlog.
Thus the two models have identical Bellman objectives. Backward induction
shows that \(V_t^*\) is convex and nondecreasing: if this holds at
\(t+1\), then \(H_t(y)=\mathbb E[V_{t+1}^*((y-D_t)^+)]\) is convex
and nondecreasing, \(U_t=C_t+H_t\) is convex, and minimizing it over
\(y\geq x^+\) gives a function constant below its smallest minimizer
and convex and nondecreasing above it. Positive costs and bounded demand
place all unconstrained minimizers of \(C_t\) in \([0,\bar D]\).
Since \(H_t\) is nondecreasing
and \(C_t\) is strictly increasing to the right of its largest minimizer,
the smallest minimizer of \(U_t\) cannot exceed \(S_t^m\).
The constrained minimizer is \(\max\{x,S_t^*\}\). \Halmos
\end{proof}

Consequently, from initial inventory in \([0,\bar D]\), optimal states
lie in \([-\bar D,\bar D]\) under backlog and \([0,\bar D]\) under
lost sales, with actions at most \(\bar D\). A policy is
\(\epsilon\)-optimal at \(x\) if \(V_1^\pi(x)-V_1^*(x)\leq\epsilon\).
Sample complexity counts scalar offline observations needed for this
guarantee at a stated confidence. 

\subsection{Demand Samples, Censored Sales, and the Usable Fraction}
\label{section:usable-fraction}
\paragraph{Periodwise demand.}
For prescribed counts \(N_t\), periodwise demand data are the fully observed,
period-labeled array \((D_t^j:t\in[T],j\in[N_t])\), with joint law
\(\bigotimes_{t=1}^T P_t^{\otimes N_t}\). Thus samples are IID within each
period and independent across periods, but the laws \(P_t\) may differ.
The scalar-sample count is \(M=\sum_tN_t\). The adaptive protocol in Theorem
\ref{theorem:bounded policy lower bound} additionally allows choosing the
next sampled period from preceding observations; each query to period \(t\)
returns a fresh draw from \(P_t\).

\paragraph{Pooled IID demands.}
When \(P_t=P\) for every \(t\), pooled IID demand data are fully observed
scalars \((D^1,\ldots,D^M)\sim P^{\otimes M}\), used jointly to estimate the
common demand law for all periods. Each draw counts once in the budget
\(M\), even when reused across periods. Our stationary results additionally
assume time-homogeneous costs. Deployment uses a fresh demand trajectory
independent of the offline data.

\paragraph{Raw sales logs.}
Raw sales logs are the period-labeled array
\(((Z_t^j,b_t^j):t\in[T],j\in[N_t])\), where the recorded boundary
\(b_t^j\geq0\) is known and \(Z_t^j=\min\{D_t^j,b_t^j\}\) is observed,
while latent demand \(D_t^j\) need not be observed. The raw sample count is
\(\sum_tN_t\): each sales--boundary pair counts as one log, including logs
later excluded by the cap qualification rule.
We assume \emph{conditional noninformative censoring}: conditional on the
complete boundary array \(\mathcal B=(b_t^j:t\in[T],j\in[N_t])\),
all latent \(D_t^j\) are mutually independent and retain law \(P_t\).
Each boundary is fixed or determined before its own demand. This assumption
is stronger than predictability: setting \(b_2=D_1\) reveals \(D_1\)
when conditioning on all boundaries, although \(b_2\) precedes \(D_2\).
Routine inventory logging that adapts to earlier fitting demands can
therefore fall outside our model.

For example, assign independent historical run \(j\) a fixed base-stock
target \(b^j\in[0,\bar D]\) and initial stock at most \(b^j\). Nonnegative
demand and zero lead time imply post-order stock \(b_t^j=b^j\) in every
period, despite inventory-dependent order quantities. Independent demands
across periods and runs give the required conditional law; targets may vary
across runs. This protocol excludes arbitrary adaptive sales logs.

Define \(F_t(u)=\mathbb P(D_t\leq u)\) and
\(F_t^-(u)=\mathbb P(D_t<u)\). For caps \(a_t\in[0,\bar D]\)
fixed independently of the fitting demands or selected using only
\(\mathcal B\), let
\begin{equation}
\label{eq:usable-fraction}
\mathcal J_t=\{j\in[N_t]:b_t^j\geq a_t\},\qquad
m_t=|\mathcal J_t|\geq rN_t,\qquad t\in[T].
\end{equation}
Here \(r\in(0,1]\) is a common lower bound on the \emph{usable fraction}.
For \(j\in\mathcal J_t\), re-censoring recovers
\(W_t^j=\min\{Z_t^j,a_t\}=\min\{D_t^j,a_t\}\), a draw from
\(W_t=\min\{D_t,a_t\}\), conditional on \(\mathcal B\).
Thus \(r\) counts logs reaching a cap, rather than fully observed demands.
With \(N_t=N\), at least \(rN\) usable draws are available per period;
a usable-sample requirement converts to a raw budget by the factor \(1/r\).

\subsection{Learning Tasks and Summary of Results}
\label{section:learning-tasks}
Recall \(\mathcal D_M\) denotes
the offline data under Section~\ref{section:usable-fraction}, and \(x\) the
specified initial inventory. We consider learning algorithms that only knows the costs and observation
protocol, but not the demand laws or optimal values. We separate the learning goals into policy learning and optimal value estimation as below.

\paragraph{Policy learning.}
The output is a feasible policy \(\widehat\pi=\mathcal A(\mathcal D_M)\)
satisfying,
\[
\mathbb P\!\left(V_{1}^{\widehat\pi}(x)-V_{1}^*(x)>\epsilon\right)
\leq\eta.
\]

\paragraph{Optimal-value estimation.}
The output is a scalar \(\widehat v=\mathcal A(\mathcal D_M)\) satisfying
\[
\mathbb P\!\left(|\widehat v-V_{1}^*(x)|>\epsilon\right)\leq\eta
\]
No policy output is required. In both tasks, probability is
over data and procedure randomization. The sample complexity question is thus: what is the \emph{minimum number of samples} required to achieve policy learning or optimal-value estimation within accuracy $\epsilon$ with high probability? We summarize our results in Table \ref{tab:information-map}.

\begin{table}[!htbp]
\caption{Summary of the sample complexity results.}
\label{tab:information-map}
\centering
\small
\setlength{\tabcolsep}{4pt}
\renewcommand{\arraystretch}{1.15}
\begin{tabular}{@{}>{\raggedright\arraybackslash}p{0.23\linewidth}>{\raggedright\arraybackslash}p{0.14\linewidth}>{\raggedright\arraybackslash}p{0.26\linewidth}>{\raggedright\arraybackslash}p{0.30\linewidth}@{}}
\toprule
Setting and initial state & Observations & Lower bound and source & Upper bound and source\\
\midrule
\multicolumn{4}{@{}l}{\textbf{A. Policy learning: output a feasible policy}}\\[3pt]
Nonstationary backlog, \(x_1=0\) & Periodwise demand &
\(\Omega(T^3\epsilon^{-2})\)\par
{\footnotesize This work: Theorem~\ref{theorem:bounded policy lower bound}.} &
\(O(T^3\epsilon^{-2})\)\par
{\footnotesize \citet{xie2026vc} benchmark; \eqref{eq:fixed-start-bound}.}\\[3pt]
Stationary, uniformly over \(x_1\in[0,1]\) & Pooled IID demands &
--- & \(O(T^2\epsilon^{-2})\)\par
{\footnotesize This work: Proposition~\ref{prop:stationary policy value}, Corollary~\ref{cor:stationary policy learning}.}\\[3pt]
Censored lost sales, \(x_1=0\), maintained coverage & Raw sales logs &
\(\Omega(T^3/(r\epsilon^2))\)\par
{\footnotesize This work: Theorem~\ref{theorem:coverage lower bound}, Proposition~\ref{prop:rounded-coverage-lower}.} &
\(O(T^3/(r\epsilon^2))\)\par
{\footnotesize This work: Theorem~\ref{theorem:coverage product erm}, using the ERM benchmark.}\\
\midrule
\multicolumn{4}{@{}l}{\textbf{B. Optimal-value estimation: output a scalar cost estimate}}\\[3pt]
Stationary backlog, \(x_1=1\) & Pooled IID demands &
\(\Omega(T^3\epsilon^{-2})\)\par
{\footnotesize This work: Theorem~\ref{theorem:stationary value lower bound}.} &
\(O(T^3\epsilon^{-2})\)\par
{\footnotesize This work: \eqref{eq:stationary-value-minimax}, inherited from \cite{xie2026vc}.}\\[3pt]
Stationary, \(x_1=0\) & Pooled IID demands &
--- & \(O(T^2\epsilon^{-2})\)\par
{\footnotesize This work: Section~\ref{section:stationary-separation}.}\\[3pt]
Censored lost sales, IID, \(x_1=0\), coverage, \(r=1\) & Raw sales logs &
\multicolumn{2}{>{\raggedright\arraybackslash}p{0.58\linewidth}@{}}{Not identifiable: no finite sample size suffices uniformly.\par
{\footnotesize This work: Proposition~\ref{prop:censored-value-nonidentification}.}}\\
\bottomrule
\end{tabular}
\par\smallskip
\begin{minipage}{\linewidth}
\footnotesize
\textit{Notes.} \(C_c=\bar D=1\), success probability \(3/4\), additive
total-cost error \(\epsilon\); lower-bound ranges are stated in the cited
results. A dash asserts no lower bound. The stationary policy row includes
\(x_1=1\); \(h=p=1\) suffices for the stationary value lower bound.
Censored nonidentification holds for \(0<\epsilon<T/32\), even
under strict coverage. Censored policy rates concern prescribed archives
under the carry-safe condition of Section~\ref{section:decision-coverage}.
\end{minipage}
\end{table}
\FloatBarrier

The rest of the paper focuses on developing these sample complexity results rigorously.

\section{Information-Theoretic Lower Bounds}
\label{section: LB}
We establish necessary information requirements for policy learning and
optimal-value estimation, including nonidentification under censoring.
Write \(\operatorname{TV}(P,Q)=\sup_A|P(A)-Q(A)|\) for total variation,
\(\operatorname{KL}(P\|Q)\) for relative entropy, and
\(\operatorname{kl}(u,v)=u\log(u/v)+(1-u)\log[(1-u)/(1-v)]\) for
Bernoulli relative entropy. A \(\operatorname{Bernoulli}(u)\) variable is
one with probability \(u\) and zero otherwise.

\subsection{Policy Learning with Full Demand Observations}
\begin{theorem}[Bounded-cost policy lower bound]
\label{theorem:bounded policy lower bound}
{There is a universal constant \(c_{\rm pol}>0\) such that the following holds. Let \(T=2K\geq4\) and \(0<\epsilon\leq T/128\). Consider the class of backlog instances with independent demands supported on \(\{0,1\}\), zero ordering cost and lead time, and
\[
h_t=p_t=1,\qquad t=1,\ldots,T.
\]
An algorithm may request observations sequentially. Before request \(m\), it chooses a period index \(A_m\in[T]\) as an arbitrary randomized function of its internal seed and all preceding indices and observations, and then receives a fresh draw from the period-\(A_m\) demand law. It stops after at most \(M\) scalar observations and returns an arbitrary feasible, possibly randomized policy. If, on every instance in the class, its returned policy is \(\epsilon\)-optimal from \(x_1=0\) with probability at least \(3/4\), then
\[
M\geq c_{\rm pol}\frac{T^3}{\epsilon^2}.
\]
The conclusion also holds with the conventional backlog action constraint
\(y_t\geq x_t\), which permits negative post-order levels.}
\end{theorem}

\begin{proof}{Proof.}
{Put \(\Delta=16\epsilon/K\leq1/4\). For each sign vector
\(\theta\in\{-1,+1\}^K\), define
\[
\begin{aligned}
D_{2i-1}&\sim\operatorname{Bernoulli}\!\left(\frac12+\theta_i\Delta\right),\\
D_{2i}&=1,\qquad i=1,\ldots,K,
\end{aligned}
\]
independently across periods. In an odd period, conditional on an action
\(y\in[0,1]\), the expected one-period cost is
\[
\frac12+\theta_i\Delta-2\theta_i\Delta y.
\]
The optimal action is \(1\) when \(\theta_i=+1\) and \(0\) when
\(\theta_i=-1\), and the corresponding excess cost is
\begin{equation}
\label{eq:lb-stage-regret}
2\Delta(1-y)\quad\hbox{or}\quad2\Delta y,
\end{equation}
respectively. The same lower bound obtained from \eqref{eq:lb-stage-regret} continues to hold for
\(y>1\), where the stage cost is affine and increasing. For \(y<0\),
the cost is \(1/2+\theta_i\Delta-y\), so a wrong classification still
costs at least \(\Delta\). In each even period,
ordering up to \(1\) incurs zero cost and leaves the next state at zero.
Moreover, every even-period cost is nonnegative. Hence the optimal total
cost is \(K(1/2-\Delta)\). Let \(\mathcal D_{\rm off}\) denote the
offline data and learner randomization. In an independent operational rollout
of a learned policy \(\pi\), define \(\widehat\theta_i=+1\) when its
period-\(2i-1\) action is at least \(1/2\), and \(-1\) otherwise.
Write \(\mathbb P_\theta^\pi\) for probabilities in this rollout. Then
\begin{equation}
\label{eq:lb-policy-regret}
R_\theta(\pi):=V_1^\pi(0)-V_1^*(0)
\geq
\Delta\sum_{i=1}^K
\mathbb P_\theta^\pi\!\left(
\widehat\theta_i\neq\theta_i\mid\mathcal D_{\rm off}\right),
\end{equation}
This argument also covers
policies that do not reset: positive carryover can constrain the next action,
but it cannot reduce any odd-period cost below \(1/2-\Delta\), and the
intervening reset cost remains nonnegative.

We now average over a uniform prior on \(\theta\). Let
\[
\mathcal H=(U,A_1,X_1,\ldots,A_M,X_M)
\]
denote the complete offline transcript, where \(U\) is the learner's random seed and \(X_m\) is the observation
returned by request \(m\); write \(\mathbb P_\theta^{\mathcal H}\) for
its law. A learner that stops early can be padded with null requests. Put
\(N_i=\sum_{m=1}^M\mathbf 1\{A_m=2i-1\}\), now a random count. Conditional on
any realized past transcript, the learner uses the same kernel to select
\(A_m\) under \(\theta\) and its \(i\)th-coordinate flip
\(\theta^{(i)}\). The conditional observation laws differ only if
\(A_m=2i-1\). Therefore the KL chain rule gives
\begin{equation}
\label{eq:lb-adaptive-kl}
\operatorname{KL}\!\left(
\mathbb P_\theta^{\mathcal H}
\middle\|
\mathbb P_{\theta^{(i)}}^{\mathcal H}\right)
\leq22\Delta^2\mathbb E_\theta[N_i],
\end{equation}
where \(\operatorname{kl}(u,v)\leq(u-v)^2/[v(1-v)]\) gives, for
\(\Delta\leq1/4\),
\[
\operatorname{kl}\!\left(\frac12-\Delta,\frac12+\Delta\right)
\leq\frac{4\Delta^2}{1/4-\Delta^2}
\leq\frac{64}{3}\Delta^2\leq22\Delta^2.
\]
The same bound holds in the reverse orientation. The independent operational
history before the period-\(2i-1\) action is a common randomization kernel
under this flip, because the \(i\)th operational demand has not yet been
realized. It consequently adds no information to the test.

Pair the signs in coordinate \(i\). Assouad's two-point step, followed by
Pinsker in both orientations of every pair, gives
\begin{equation}
\label{eq:lb-adaptive-assouad}
\begin{aligned}
2^{-K}\sum_\theta\sum_{i=1}^K
\mathbb P_\theta(\widehat\theta_i\neq\theta_i)
&\geq\frac K2-
2^{-K}\sum_{i=1}^K\sum_{\theta:\theta_i=+1}
\operatorname{TV}\!\left(
\mathbb P_\theta^{\mathcal H},
\mathbb P_{\theta^{(i)}}^{\mathcal H}\right)\\
&\geq\frac K2-\frac12\sqrt{11KM\Delta^2}.
\end{aligned}
\end{equation}
For the last step, average the two directed Pinsker bounds for every pair and
apply Cauchy--Schwarz to
\(2^{-K}\sum_\theta\sum_i\sqrt{\mathbb E_\theta[N_i]}\), using
\(\sum_iN_i\leq M\). If \(M\leq K/(176\Delta^2)\),
\eqref{eq:lb-adaptive-assouad} implies an average Hamming error of at
least \(3K/8\). Define the bounded surrogate
\(Z_\theta=\Delta\sum_i
\mathbb P_\theta^\pi(\widehat\theta_i\neq\theta_i\mid\mathcal D_{\rm off})\).
Then \(0\leq Z_\theta\leq K\Delta=16\epsilon\), \(Z_\theta\leq
R_\theta(\pi)\) by \eqref{eq:lb-policy-regret}, and the preceding display gives
\(\mathbb E Z_\theta\geq6\epsilon\) after averaging over signs. If the
algorithm succeeded with probability \(3/4\) for every sign vector, the same
average would instead be at most
\((3/4)\epsilon+(1/4)16\epsilon=19\epsilon/4\), a contradiction. Hence
\[
M>\frac{K}{176\Delta^2}
=\frac{K^3}{45056\epsilon^2}
=\frac{T^3}{360448\epsilon^2},
\]
which proves the claim. \Halmos}
\end{proof}

The lower bound matches the \(O(T^3\epsilon^{-2})\) scalar-observation
upper order from Product ERM \citep{xie2026vc}. With per-period accuracy
\(\alpha=\epsilon/T\), the fixed-confidence order is
\(\Theta(T\alpha^{-2})\), or \(\Theta(\alpha^{-2})\) balanced
trajectories of \(T\) observations each.

\subsection{Policy Learning with Censored Observations}
\label{section:censored-policy-lower}
The next result isolates the sample-size cost of a small usable fraction
\(r\), as defined in Section~\ref{section:usable-fraction}. In its
construction, logs whose boundaries reach the cap \(7/8\) reveal the
unknown odd-period demand probabilities, whereas the remaining low-boundary
logs carry no information about them.

\begin{theorem}[Statistical price of limited decision coverage]
\label{theorem:coverage lower bound}
{There is a universal constant \(c_{\rm cov}>0\) such that the
following holds. Let \(T=2K\geq4\), \(r\in(0,1]\), and
\(0<\epsilon\leq T/512\). There is a family with a common
decision-coverage cap \(7/8\) satisfying
\(F_t^-(7/8)>1/2\) in every period and retaining positive above-cap demand
mass in the even periods. Every even-period boundary is \(7/8\).
In odd period \(2i-1\), exactly \(rN_{2i-1}\) of the predetermined logged
observations have boundary \(7/8\), while all remaining boundaries lie below
\(1/4\); assume for simplicity that the relevant products with \(r\) are
integers, since rounding changes only constants. Even when \(h_t=p_t=1\),
any algorithm that returns an
\(\epsilon\)-optimal lost-sales policy from \(x_1=0\) with probability at
least \(3/4\) on every instance must use
\[
M\geq c_{\rm cov}\frac{T^3}{r\epsilon^2}
\]
total logged observations.}
\end{theorem}

\begin{proof}{Proof.}
{Let \(a=1/4\), \(b=1/2\), \(w=b-a=1/4\), and
\(\Delta=16\epsilon/(Kw)\leq1/4\). For
\(\theta\in\{-1,+1\}^K\), set
\[
\begin{aligned}
\mathbb P_\theta(D_{2i-1}=b)&=\frac12+\theta_i\Delta,
&\mathbb P_\theta(D_{2i-1}=a)&=\frac12-\theta_i\Delta,\\
\mathbb P(D_{2i}=3/4)&=\frac34,
&\mathbb P(D_{2i}=1)&=\frac14.
\end{aligned}
\]
The odd-period optimum is \(b\) for a positive sign and \(a\) for a
negative sign. Define \(\widehat\theta_i=+1\) if the deployed odd-period
action is at least \((a+b)/2\), and \(-1\) otherwise. A wrong classification increases the odd-period
expected cost by at least \(\Delta w\). In an even period, the minimum
expected cost is \(1/16\), attained by ordering to \(3/4\); both possible
demands then reset the lost-sales state to zero. The odd optimal action
leaves a state of at most \(w\), so this reset action is feasible. Every
odd and even stage cost is bounded below by its corresponding one-period
minimum, and the described policy attains all those minima. Thus the
analogue of \eqref{eq:lb-policy-regret} holds with \(\Delta\) replaced by
\(\Delta w\).

For low-boundary observations choose a boundary \(\ell<a\); then
\(\min\{D_{2i-1},\ell\}=\ell\) under both signs and the observation has zero
information. Set every revealing odd-period boundary to \(7/8\), which
reveals the two-point demand. Let \(m_i=rN_{2i-1}\) be the number
of revealing observations in coordinate \(i\), and let
\(\operatorname{TV}_i\) be the maximum transcript total variation over
pairs differing only in coordinate \(i\). Pinsker's inequality gives
\(\operatorname{TV}_i\leq\sqrt{11}\Delta\sqrt{m_i}\); Cauchy--Schwarz and
\(\sum_i m_i\leq rM\) yield
\(\sum_i\operatorname{TV}_i\leq\sqrt{11KrM\Delta^2}\).
Pairing signs under the
uniform prior makes the factor \(1/2\) explicit:
\[
\begin{aligned}
2^{-K}\sum_\theta\sum_{i=1}^K
\mathbb P_\theta(\widehat\theta_i\neq\theta_i)
&\geq\frac12\sum_{i=1}^K(1-\operatorname{TV}_i)\\
&\geq\frac K2-\frac12\sqrt{11KrM\Delta^2}.
\end{aligned}
\]
If \(M\leq K/(176r\Delta^2)\), this is at least \(3K/8\).
The bounded-surrogate argument in Theorem
\ref{theorem:bounded policy lower bound}, with per-error loss
\(\Delta w\), therefore forces
\[
M>\frac{K}{176r\Delta^2}
=\frac{K^3w^2}{45056r\epsilon^2}
=\frac{T^3}{5767168r\epsilon^2}.
\]
Finally, \(F_{2i-1}^-(7/8)=1\) and \(F_{2i}^-(7/8)=3/4>1/2\),
with even-period mass \(1/4\) at \(1>7/8\): strict decision coverage
holds despite a genuinely unobserved tail. \Halmos}
\end{proof}

\paragraph{The prescribed-archive setting.}
Normalize \(C_c=\bar D=1\) and \(x_1=0\). For each integer \(N\), supply
\(N\) raw logs per period with prescribed boundaries and nondecreasing caps
\(0\leq a_t\leq1\) satisfying \(F_t^-(a_t)>p_t/(h_t+p_t)\) and
\(m_t\geq\lceil rN\rceil\). The caps are known and supplied to the learner,
who observes these arrays but cannot
choose them, and latent demands satisfy Section~\ref{section:usable-fraction}.
Let \(N_{\rm cov}^*(T,\epsilon,r)\) be the least \(N\) permitting one
learner to be \(\epsilon\)-optimal with probability at least \(3/4\),
uniformly over these demand laws, known costs, and arrays. Write
\(\mathcal M_{\rm cov}^*=TN_{\rm cov}^*\).

\begin{proposition}[Censored policy lower bound with integer counts]
\label{prop:rounded-coverage-lower}
For even \(T\geq4\), \(r\in(0,1]\), and \(0<\epsilon\leq T/512\),
\(\mathcal M_{\rm cov}^*(T,\epsilon,r)>T^3/(46137344r\epsilon^2)\).
\end{proposition}
\begin{proof}{Proof.}
Use the family of Theorem
\ref{theorem:coverage lower bound} with \(K=T/2\), \(w=1/4\), and
\(\Delta=8\epsilon/(Kw)\leq1/8\). Give each odd period exactly
\(m=\lceil rN\rceil\) revealing logs, and each even period all
revealing logs. Splitting off one observation, tensorization of total
variation and Pinsker give, for every coordinate flip,
\[
\operatorname{TV}_i\leq2\Delta+\sqrt{11\Delta^2(m-1)}
\leq\frac14+\sqrt{11\Delta^2rN}.
\]
If \(rN\leq1/(11264\Delta^2)\), the mean Hamming error is at least
\(23K/64\). The bounded regret surrogate in that proof is at most
\(8\epsilon\) and has mean at least \(23\epsilon/8\). Success
probability \(3/4\) on every instance would instead bound its mean by
\(3\epsilon/4+8\epsilon/4=11\epsilon/4\), a contradiction.
Thus \(TN>T^3/(46137344r\epsilon^2)\), including integer rounding.
\Halmos
\end{proof}

\subsection{Optimal-Value Estimation with Inherited Inventory}
\label{section:value-lower}
The next family has a common optimal policy, but its value is sensitive to rare demand.

\begin{theorem}[Stationary optimal-value lower bound]
\label{theorem:stationary value lower bound}
{For every \(T\geq8\), \(0<\epsilon\leq T/1024\), and
\(0<\eta<1/2\), there is a two-point stationary backlog family such that any
estimator that estimates \(V_1^*(1)\) to additive error \(\epsilon\) with
probability at least \(1-\eta\) on both instances requires
\[
M\geq
\frac{T^3}{360448\epsilon^2}
(1-2\eta)\log\!\left(\frac{1-\eta}{\eta}\right)
\]
pooled IID demand observations. In particular, for \(\eta\leq1/4\), this is
\(\Omega(T^3\epsilon^{-2}\log(1/\eta))\). The instances have
\(h_t=p_t=1\), \(D_t\in\{0,1\}\), and a base-stock-zero policy that is
optimal throughout the family and therefore requires no learning.}
\end{theorem}

\begin{proof}{Proof.}
{Let \(\mathbb P(D_t=0)=\rho\) IID over time. For all \(\rho>1/2\),
the myopic base-stock level is zero, so Lemma \ref{lemma:myopic base stock}
implies that base stock zero is dynamically optimal. Starting from one unit,
the unit is held until the first positive demand and, after it is consumed,
each subsequent positive demand is backlogged. Direct summation gives
\begin{equation}
\label{eq:val-formula}
v_T(\rho):=V_1^*(1)
=T(1-\rho)+(2\rho-1)\sum_{k=0}^{T-1}\rho^k.
\end{equation}
Writing \(S_T(\rho)=\sum_{k=0}^{T-1}\rho^k\),
\begin{equation}
\label{eq:val-derivative}
v_T'(\rho)=-T+2S_T(\rho)+(2\rho-1)S_T'(\rho).
\end{equation}
On
\(I_T=[1-5/(8T),1-3/(8T)]\), Bernoulli's inequality gives
\(\rho^{T-1}\geq3/8\). Hence
\(S_T\geq3T/8\), while summing the terms with
\(k\geq\lceil T/2\rceil\) gives
\(S_T'\geq9T^2/128\). Also \(2\rho-1\geq27/32\). Substitution into
\eqref{eq:val-derivative} yields \(v_T'(\rho)\geq T^2/64\) for \(T\geq8\).

Set
\[
\rho_0=1-\frac1{2T},\qquad
\delta=\frac{128\epsilon}{T^2},\qquad
\rho_\pm=\rho_0\pm\delta.
\]
The assumed range of \(\epsilon\) puts both parameters in \(I_T\), and the
mean-value theorem gives
\(|v_T(\rho_+)-v_T(\rho_-)|\geq4\epsilon\). Meanwhile,
the bound
\(\operatorname{kl}(p,q)\leq(p-q)^2/[q(1-q)]\), together with
\(q\geq1/2\) and \(1-q\geq3/(8T)\) throughout \(I_T\), gives
\begin{equation}
\label{eq:val-kl}
\max\!\left\{
\operatorname{kl}(\rho_-,\rho_+),
\operatorname{kl}(\rho_+,\rho_-)
\right\}
\leq22T\delta^2.
\end{equation}
An estimator accurate to \(\epsilon\) on both instances induces, by
thresholding at the midpoint of the two values, a test with error at most
\(\eta\) under each instance. Data processing for relative entropy gives
\[
M\operatorname{kl}(\rho_+,\rho_-)
\geq
\operatorname{kl}(1-\eta,\eta)
=(1-2\eta)\log\!\left(\frac{1-\eta}{\eta}\right).
\]
Consequently,
\[
M\geq
\frac{T^3}{360448\epsilon^2}
(1-2\eta)\log\!\left(\frac{1-\eta}{\eta}\right).
\]
For \(\eta\leq1/4\), the last factor is bounded below by a universal
positive constant times \(\log(1/\eta)\).
The policy is identical and known on the two instances; only its value is
hard to estimate. \Halmos}
\end{proof}

Section~\ref{section:stationary-policy} gives quadratic policy learning on
this same class and initial state: identifying the policy is easier.

\subsection{Optimal-Value Nonidentification under Censoring}
\begin{proposition}[Identical sales and optimal policies, different values]
\label{prop:censored-value-nonidentification}
For every \(T\geq1\) and \(0<\epsilon<T/32\), optimal-value estimation
from censored lost-sales logs has minimax failure probability at least
\(1/2\) for every finite archive size, already with IID demand,
\(h_t=p_t=\bar D=1\), \(x_1=0\), strict coverage, and \(r=1\).
\end{proposition}
\begin{proof}{Proof.}
Consider IID lost-sales demand with \(h_t=p_t=1\), \(\bar D=1\),
\(x_1=0\), and every logging boundary and cap equal to \(a=1/2\). Set
\(P=\tfrac34\delta_0+\tfrac14\delta_{3/4}\) and
\(Q=\tfrac34\delta_0+\tfrac14\delta_1\).
Both laws give sales zero with probability \(3/4\) and \(1/2\) otherwise.
Every finite archive therefore has the same law, although all logs qualify
and \(r=1\). Strict coverage holds:
\(P(D<a)=Q(D<a)=3/4>1/2\). Zero is the unique myopic minimizer under
both laws, so Lemma~\ref{lemma:myopic base stock} makes the
zero-base-stock policy optimal. Writing \(V_{1,P}^*\) and \(V_{1,Q}^*\)
for the optimal values under the respective laws gives
\[
V_{1,P}^*(0)=T\mathbb E_P D=\frac{3T}{16},\qquad
V_{1,Q}^*(0)=T\mathbb E_Q D=\frac T4.
\]
The common truncated optimum is \(T/8\); the missing tail constants are
\(T/16\) and \(T/8\). Thus, for \(0<\epsilon<T/32\) and every finite
archive size,
\[
\inf_{\widehat V}\max_{R\in\{P,Q\}}
\mathbb P_R\!\left(|\widehat V-V_{1,R}^*(0)|>\epsilon\right)\geq\frac12,
\]
where \(\widehat V\) ranges over all estimators from the sales logs,
including randomized ones. The two success intervals are disjoint, while
the estimator has the same output law, so their probabilities sum to at most
one. The decision is identified, but its absolute value is not.
\Halmos
\end{proof}

\section{A Full-Observation Benchmark}
\label{section: VI}
With fully observed demand, period-specific empirical marginals yield a
benchmark for both policy learning and valuation. We first state its guarantee
and then establish the uniform-convergence argument used again in
Section~\ref{section: censored}.
\subsection{Product ERM and the Empirical-DP Guarantee}
\label{section:empirical-dp}
We use the Product ERM guarantee of \citet{xie2026vc} as an upper-bound
benchmark. The same uniform estimation event controls both policy regret
and optimal-value error.

\noindent\textbf{Product empirical risk minimization (Product ERM).}
Given \(N\) fully observed demand samples per period, mutually independent
within and across periods, define
\[
\widehat P_t=\frac1N\sum_{j=1}^N\delta_{D_t^j},\qquad
\widehat P^\times=\bigotimes_{t=1}^T\widehat P_t.
\]
Here \(\delta_d\) is a point mass. Let
\(\Pi_{\rm BS}=\{\pi^S:S\in[0,\bar D]^T,
\pi_t^S(x)=\max\{x,S_t\}\}\). For initial stock \(x_1\), let
\(\ell_{x_1}(\pi;d_{1:T})\) be the realized total holding and shortage
cost under the dynamics of Section~\ref{section: inventory preliminaries}.
Product ERM selects
\[
\widehat\pi^\times\in\argmin_{\pi\in\Pi_{\rm BS}}
\mathbb E_{\widehat P^\times}[\ell_{x_1}(\pi;D_{1:T})].
\]
Replacing each demand law in the Bellman recursion by \(\widehat P_t\)
implements this minimization without enumerating trajectories. This is
\emph{empirical DP}, or periodwise \emph{sample average approximation
(SAA)}. We use its exact DP-optimal representative
\(\widehat\pi^{\rm DP}\), optimal under the empirical product law from
every admissible state. Denote its thresholds by \(\widehat S_t\) and
its optimal values by \(\widehat V_t^{\rm DP}\). Unlike trajectory ERM, Product ERM integrates
over the product of empirical marginals. Recombination creates no new
independent observations: the scalar sample budget is \(M=TN\).

\begin{corollary}[Fixed-start empirical-DP benchmark]
\label{prop:empirical-dp-joint}
In the fully observed backlog model with \(N\) samples per period, there
is a universal constant \(C_0\) such that, with probability at least
\(1-\eta\), simultaneously for every \(x\in[-\bar D,\bar D]\),
\[
|\widehat V_1^{\rm DP}(x)-V_1^*(x)|\leq b_1,\qquad
V_1^{\widehat\pi^{\rm DP}}(x)-V_1^*(x)\leq2b_1,
\]
where
\begin{equation}
\label{eq:fixed-start-bound}
b_1=\frac{C_c\bar D T}{\sqrt N}
\left(C_0+\sqrt{\frac{\log(1/\eta)}2}\right).
\end{equation}
For \(0<\eta<1\) and \(0<\epsilon\leq C_c\bar D T\),
\(N=O(C_c^2\bar D^2T^2\epsilon^{-2}[1+\log(1/\eta)])\) suffices
for both errors to be at most \(\epsilon\). 
\end{corollary}

\subsection{Horizon-Free Uniform Convergence and Product Transfer}
\label{section:product-transfer}
\paragraph{The horizon-free ingredient and its scope.}
Corollary~4 of \citet{xie2026vc} bounds expected one-sided uniform
deviation by \(C(L+1)U/\sqrt N\), for positive integers \(d,N\),
\(U>0\), and integer lead time \(L\geq0\). Starting empty, with zero
ordering cost, demands in \([0,U]\), common costs in \([0,1]\) and
thresholds in \([0,(L+1)U]\), loss averages costs in periods
\(L+1,\ldots,L+d\). The \(N\) trajectories are IID from any joint
demand law; \(C\) is universal.
For our zero-lead-time application, let \(Q\) be any law on
\([0,\bar D]^d\), \(Z_1,\ldots,Z_N\) IID trajectories from \(Q\),
and \(\ell_S(z)=d^{-1}\sum_{j=1}^d c_j(Y_j(S,z),z_j)\), where
\(Y_1=S_1\) and \(Y_{j+1}=\max\{S_{j+1},Y_j-z_j\}\).
Write \(Qf=\int f\,dQ\).
Their symmetrization argument (Section~3.1, p.~11), with periodwise
\(C_c\)-Lipschitz costs, yields the precise consequence used here:
\begin{equation}
\label{eq:horizon-free-ingredient}
\mathbb E\sup_{S\in[0,\bar D]^d}
\left|Q\ell_S-\frac1N\sum_{i=1}^N\ell_S(Z_i)\right|
\leq\frac{C_0C_c\bar D}{\sqrt N}.
\end{equation}
Here \(C_0\) absorbs the two-sided factor, with no \(d\), \(\log d\),
or \(\log N\) dependence. Their normalized inventory-level class has
scale-sensitive dimension at most \(2/\gamma+1\) at margin
\(0<\gamma\leq1\), by ordering last replenishment
times and using nonnegative demands. Its finite entropy integral gives a
period-uniform contraction bound; averaging preserves the rate even though
ordinary pseudodimension can grow. Independent nonidentical demands are
included; independence is needed separately for product transfer. Nonnegative
thresholds give identical post-order recursions and stage costs under lost
sales, also for truncated demands in \([0,a_j]\subseteq[0,\bar D]\).
This reduction extends \eqref{eq:horizon-free-ingredient} to that class and
to tails of length \(d=T-t+1\).

Using the cyclic partition in Section~4.5 of \citet{xie2026vc}, we obtain
the following high-probability product-empirical transfer.

\begin{lemma}[High-probability product-empirical transfer]
\label{lemma:product empirical concentration}
{Let \(d,n\) be positive integers and \(P=\bigotimes_{t=1}^{d}P_t\).
For each coordinate let
\(X_{t,1},\ldots,X_{t,n}\) be IID from \(P_t\), independently across
coordinates; assume the displayed suprema are measurable. Write \(\widehat P_t\) for the corresponding empirical
marginal and \(\widehat P^\times=\bigotimes_{t=1}^{d}\widehat P_t\).
For a function \(f\), write \(Pf=\int f\,dP\).
Let \(\mathcal F\) be a class of measurable functions whose common range
has length at most \(R>0\), and define
\[
G_n^\times:=\sup_{f\in\mathcal F}
|Pf-\widehat P^\times f|.
\]
For a deterministic bound \(\mu_n\geq0\), if the ordinary IID empirical process based on
\(Z_1,\ldots,Z_n\stackrel{\mathrm{iid}}\sim P\) satisfies
\[
\mathbb E\!\left[
\sup_{f\in\mathcal F}|Pf-n^{-1}\textstyle\sum_{i=1}^nf(Z_i)|
\right]\leq\mu_n,
\]
then, for every \(0<\delta<1\),
\[
G_n^\times\leq
\mu_n+R\sqrt{\frac{\log(1/\delta)}{2n}}
\]
with probability at least \(1-\delta\). Consequently, if
\(\widehat f\in\arg\min_{f\in\mathcal F}\widehat P^\times f\) and
\(f^*\in\arg\min_{f\in\mathcal F}Pf\), then
\[
P\widehat f-Pf^*\leq
2\mu_n+2R\sqrt{\frac{\log(1/\delta)}{2n}}
\]
on the same event.}
\end{lemma}

\begin{proof}{Proof.}
{Index the \(n^{d-1}\) cyclic groups by
\(a=(a_2,\ldots,a_d)\in\{0,\ldots,n-1\}^{d-1}\). Group \(a\) consists of
the \(n\) trajectories
\[
Z_i^a=(X_{1,i},X_{2,i+a_2},\ldots,X_{d,i+a_d}),
\qquad i=1,\ldots,n,
\]
where indices are taken modulo \(n\). Within a fixed group these
trajectories are IID from \(P\), and averaging over all groups gives
exactly \(\widehat P^\times\). If
\(G_a=\sup_f|Pf-n^{-1}\sum_i f(Z_i^a)|\), the triangle inequality gives
\[
G_n^\times\leq n^{-(d-1)}\sum_aG_a.
\]
For \(\lambda>0\), changing one trajectory in a fixed group changes \(G_a\) by at most
\(R/n\). The bounded-differences moment-generating-function bound thus
gives
\[
\mathbb E\exp\{\lambda(G_a-\mathbb EG_a)\}
\leq\exp\!\left(\frac{\lambda^2R^2}{8n}\right).
\]
The groups need not be independent. Pointwise convexity of the exponential
and \(\mathbb EG_a\leq\mu_n\) imply
\[
\mathbb E e^{\lambda G_n^\times}
\leq n^{-(d-1)}\sum_a\mathbb E e^{\lambda G_a}
\leq\exp\!\left(\lambda\mu_n+\frac{\lambda^2R^2}{8n}\right).
\]
Chernoff optimization proves the first claim. The second is the standard
ERM inequality \(P\widehat f-Pf^*\leq2G_n^\times\). \Halmos}
\end{proof}

\begin{proof}{Proof of Corollary~\ref{prop:empirical-dp-joint}.}
For \(x\in[-\bar D,\bar D]\) and a tail of length \(d_t=T-t+1\), set
\[
\Pi_{t,x}=\{S\in[0,\bar D]^{d_t}:S_t\geq\max\{x,0\}\}.
\]
Let \(J_t(S)\) and \(\widehat J_t^\times(S)\) be the true and
product-empirical total tail costs from zero. On this subclass they also
equal the costs from \(x\), since replenishment is free. Replacing the
first threshold by its maximum with \(x\) leaves its deployed action
unchanged, so
\[
V_t^*(x)=\inf_{S\in\Pi_{t,x}}J_t(S),\qquad
\widehat V_t^{\rm DP}(x)=\inf_{S\in\Pi_{t,x}}\widehat J_t^\times(S).
\]
Normalized losses lie in \([0,C_c\bar D]\). For \(t=1\), Lemma
\ref{lemma:product empirical concentration} and the expected-deviation
bound above give \(\sup_S|J_1(S)-\widehat J_1^\times(S)|\leq b_1\)
with probability at least \(1-\eta\). Every \(\Pi_{1,x}\) is a subclass
of this same class, so no union over states is needed. Infimum stability
gives value error at most \(b_1\). Empirical DP attains the empirical
infimum, hence its true cost is at most
\(\widehat V_1^{\rm DP}(x)+b_1\leq V_1^*(x)+2b_1\). The stated
sample order makes \(2b_1\leq\epsilon\). \Halmos
\end{proof}

\paragraph{Extension to all starting periods and states.}
Apply the same argument to each tail with failure budget \(\delta/T\).
A union bound gives, simultaneously for all \(t\) and \(x\), value error
at most \(b_t\) and policy gap at most \(2b_t\), where
\[
b_t=\frac{C_c\bar D d_t}{\sqrt N}
\left(C_0+\sqrt{\frac{\log(T/\delta)}2}\right).
\]
Thus, for a sufficiently large universal \(C\),
\(N\geq C C_c^2\bar D^2T^2\epsilon^{-2}(1+\log(T/\delta))\) implies
\begin{equation}
\label{eq:empirical-dp-joint}
\sup_{t,x}\bigl(V_t^{\widehat\pi^{\rm DP}}(x)-V_t^*(x)\bigr)\leq\epsilon,
\qquad
\sup_{t,x}|\widehat V_t^{\rm DP}(x)-V_t^*(x)|\leq\epsilon
\end{equation}
with probability at least \(1-\delta\), for
\(t\in[T]\), \(x\in[-\bar D,\bar D]\), and
\(0<\epsilon\leq C_c\bar D T\). The extra \(\log T\) serves only the
simultaneous guarantee across starting periods.

\subsection{Stationary Policy Learning}
\label{section:stationary-policy}
Stationarity allows all observations to estimate a common demand law.
Pooling reduces the policy-learning upper order from cubic to quadratic in
the horizon, as the following newsvendor reduction shows.
\begin{proposition}[Uniform stationary policy bound]
\label{prop:stationary policy value}
{Suppose demand is IID on \([0,\bar D]\), the holding and shortage
coefficients \(h,p>0\) are constant over time, and ordering has zero fixed and per-unit
costs and zero lead time. Let
\[
g(y)=\mathbb E\!\left[h(y-D)^++p(D-y)^+\right],
\qquad s^*:=\min\arg\min_{y\in[0,\bar D]}g(y),
\]
and let \(\pi_s\) be the stationary base-stock policy with level
\(s\in[0,\bar D]\). In either the backlog or lost-sales model,
\(\pi_{s^*}\) is optimal and
\[
\sup_{x_1\in[0,\bar D]}
\left[V_1^{\pi_s}(x_1)-V_1^*(x_1)\right]
\leq T\left[g(s)-g(s^*)\right].
\]
If \(x_1\leq\min\{s,s^*\}\), equality holds:
\[
V_1^{\pi_s}(x_1)-V_1^*(x_1)
=T\left[g(s)-g(s^*)\right].
\]}
\end{proposition}

\begin{proof}{Proof.}
{First, \(\pi_{s^*}\) is optimal from every nonnegative state. If the
current state is at most \(s^*\), ordering to \(s^*\) attains the universal
per-period lower bound \(g(s^*)\), now and after every future demand. If the
current state exceeds \(s^*\), convexity makes \(g\) nondecreasing on
\([s^*,\infty)\). The continuation value is nondecreasing in the state by
backward induction, so ordering above the current state cannot improve
either the current or continuation term. Thus the action
\(\max\{x,s^*\}\) is optimal.

For the uniform comparison, couple \(\pi_s\) and \(\pi_{s^*}\) with the
same demand sequence. Their order-up-to levels agree until the first time
one policy orders to the larger of \(s\) and \(s^*\). Thereafter the policy
with the larger threshold orders to that threshold in every period, while
the other policy's order-up-to level remains between the two thresholds.
Because \(g\) is convex and \(s^*\) minimizes it, the candidate policy's
conditional expected excess stage cost is at most
\(g(s)-g(s^*)\) in every period. Summing proves the uniform inequality.
When \(x_1\leq\min\{s,s^*\}\), each policy orders to its threshold in every
period, so its expected cost is respectively \(Tg(s)\) and \(Tg(s^*)\),
which proves equality. \Halmos}
\end{proof}

\begin{corollary}[Pooled empirical newsvendor policy]
\label{cor:stationary policy learning}
{Under the conditions of Proposition
\ref{prop:stationary policy value}, let \(\widehat s_M\) minimize the
empirical cost
\(\widehat g_M(s)=M^{-1}\sum_{j=1}^M[h(s-D^j)^++p(D^j-s)^+]\)
from pooled IID observations \(D^1,\ldots,D^M\).
For every \(0<\eta<1\), with probability at least \(1-\eta\),
\[
\sup_{x_1\in[0,\bar D]}
\left[V_1^{\pi_{\widehat s_M}}(x_1)-V_1^*(x_1)\right]
\leq
2T\max\{h,p\}\bar D
\sqrt{\frac{\log(2/\eta)}{2M}}.
\]
Consequently, \(M\geq2\max\{h,p\}^2\bar D^2T^2\epsilon^{-2}\log(2/\eta)\)
suffices for \(\epsilon\)-optimality.}
\end{corollary}

\begin{proof}{Proof.}
{Let \(F\) and \(\widehat F_M\) be the population and empirical
distribution functions. For \(s\in[0,\bar D]\), the CDF identity
\(g(s)=h\int_0^s F(u)\,du+p\int_s^{\bar D}[1-F(u)]\,du\),
and its empirical counterpart, give
\[
\sup_{s\in[0,\bar D]}|\widehat g_M(s)-g(s)|
\leq\max\{h,p\}\bar D\|\widehat F_M-F\|_\infty.
\]
The empirical-minimizer inequality therefore bounds
\(g(\widehat s_M)-g(s^*)\) by twice this quantity. The
Dvoretzky--Kiefer--Wolfowitz inequality and Proposition
\ref{prop:stationary policy value} give the result. \Halmos}
\end{proof}

\subsection{Optimal-Value Estimation and the Stationary Separation}
\label{section:stationary-separation}
Let \(\mathcal M_{\rm val,stat}^*(T,\epsilon,\eta;1)\) be the least
number of pooled IID demand observations for which one estimator of
\(V_1^*(1)\) is \(\epsilon\)-accurate with probability at least
\(1-\eta\), uniformly over backlog instances with demand on \([0,1]\)
and known, time-homogeneous costs \(0<h,p\leq1\). For \(T\geq8\),
\(0<\epsilon\leq T/1024\), and \(0<\eta\leq1/4\),
\begin{equation}
\label{eq:stationary-value-minimax}
\mathcal M_{\rm val,stat}^*(T,\epsilon,\eta;1)
=\Theta\!\left(T^3\epsilon^{-2}\log(1/\eta)\right).
\end{equation}
Indeed, partition pooled IID data into \(T\) independent blocks of
\(N=\lfloor M/T\rfloor\) observations, one per period. The fixed-start
bound \eqref{eq:fixed-start-bound} gives the upper order, without a
\(\log T\) factor; integer rounding adds only \(O(T)\) samples.
Theorem~\ref{theorem:stationary value lower bound} gives the lower order
on this same class. At fixed confidence, on this same class and initial state,
\[
\boxed{\begin{aligned}
\text{Policy learning at }x_1=1:&\quad O(T^2\epsilon^{-2}),\\
\text{Optimal-value estimation at }x_1=1:&\quad \Theta(T^3\epsilon^{-2}).
\end{aligned}}
\]
The policy line follows from Corollary~\ref{cor:stationary policy learning}.

\paragraph{Empty-start optimal-value estimation.}
Under the stationary assumptions of Proposition
\ref{prop:stationary policy value}, starting from \(x_1=0\) gives
\(V_1^*(0)=T\min_{s\in[0,\bar D]}g(s)\). Define
\(\widehat V_1(0)=T\min_s\widehat g_M(s)\), where
\(\widehat g_M\) is the pooled empirical newsvendor cost. Infimum stability
and the identity
\(g(s)=h\int_0^s F(u)\,du+p\int_s^{\bar D}(1-F(u))\,du\) give
\[
|\widehat V_1(0)-V_1^*(0)|
\leq T\sup_s|\widehat g_M(s)-g(s)|
\leq T\max\{h,p\}\bar D\,\|\widehat F_M-F\|_\infty.
\]
The Dvoretzky--Kiefer--Wolfowitz inequality therefore bounds this error by
\(T\max\{h,p\}\bar D\sqrt{\log(2/\eta)/(2M)}\) with probability
at least \(1-\eta\). Empty-start valuation thus uses
\(O(T^2\epsilon^{-2}\log(1/\eta))\) pooled observations at bounded
cost and demand scales. The inherited unit in Theorem~\ref{theorem:stationary value lower bound}
creates the separation as desired.

\section{Learning from Censored Sales}
\label{section: censored}
We recover fully observed truncated demands from qualified sales logs.
Carry-safe coverage makes the truncated problem decision-equivalent to the
original one, allowing the benchmark in Section~\ref{section: VI} to attain
the censored policy-learning limit.

\subsection{Why Coverage Is Necessary}
An observable boundary can limit decision identification even with unlimited
data \citep{bu2020offline,hssaine2024censored}. The following single-period
baseline motivates our carry-safe multistage reduction.

\begin{proposition}[Single-period nonidentification]
\label{prop:unobserved-tail}
Let \(T=1\), \(h,p>0\), \(0\leq\lambda<\bar D\), and let
\(\mathcal P\) be all demand laws supported on \([0,\bar D]\).
For \(P\in\mathcal P\), define
\(C_P(y)=\mathbb E_{D\sim P}[h(y-D)^++p(D-y)^+]\).
A learner \(\mathcal A\) observes \(Z_{1:N}=(Z_1,\ldots,Z_N)\),
where \(Z_j=\min\{D^j,\lambda\}\) and \(D^j\) are IID from \(P\),
and returns \(\widehat y_{\mathcal A}(Z_{1:N})\in[0,\bar D]\).
For every \(N\),
\[
\inf_{\mathcal A}\sup_{P\in\mathcal P}
\mathbb E_{P}\!\left[C_P(\widehat y_{\mathcal A})
-\min_{y\in[0,\bar D]}C_P(y)\right]
\geq\frac{\min\{h,p\}}2(\bar D-\lambda),
\]
where the expectation includes the sample and learner randomization.
\end{proposition}

\begin{proof}{Proof.}
Take the point masses \(P=\delta_\lambda\) and \(Q=\delta_{\bar D}\).
Every observed sale equals \(\lambda\), so the learner's output has a
common law. Both optimal costs are zero. For every \(y\in[0,\bar D]\),
\[
C_P(y)+C_Q(y)\geq\min\{h,p\}(\bar D-\lambda).
\]
For \(y\in[\lambda,\bar D]\), this follows from
\(h(y-\lambda)+p(\bar D-y)\); for \(y<\lambda\), the second cost
alone suffices. Averaging over the common output law and taking the larger
of the two expected regrets proves the bound. \Halmos
\end{proof}

\subsection{Carry-Safe Decision Coverage}
\label{section:decision-coverage}
Choose \(\lambda_t\in[0,\bar D]\) before observing fitting demands, or
using only the complete boundary array under Section~\ref{section:usable-fraction}.
An independent pilot is also allowed if, conditional on the pilot, design
information, and complete boundaries, fitting demands remain independent
with laws \(P_t\). Selection from fitting demands is not covered. Define
\begin{equation}
\label{eq:coverage-cap}
a_t:=\min_{s=t,\ldots,T}\lambda_s,\qquad
q_t:=\frac{p_t}{h_t+p_t},\qquad t\in[T].
\end{equation}
These are the largest nondecreasing caps with \(a_t\leq\lambda_t\).
Thus stock carried from an action below \(a_t\) cannot exceed
\(a_{t+1}\). The usable set is directly
\(\mathcal J_t=\{j\in[N_t]:b_t^j\geq a_t\}\), with count
\(m_t=|\mathcal J_t|\), as in \eqref{eq:usable-fraction}.

\begin{theorem}[Carry-safe decision coverage]
\label{theorem:decision coverage}
Suppose the conditional noninformative-censoring assumption holds,
\(x_1\in[0,a_1]\), and
\begin{equation}
\label{eq:coverage-condition}
F_t^-(a_t)>q_t,\qquad t\in[T].
\end{equation}
Let \(X_t^*\) and \(Y_t^*\) denote the inventory before and after ordering
under \(\pi^{S^*}\). Then \(S_t^*<a_t\), and
this optimal lost-sales trajectory satisfies \(X_t^*\leq a_t\) and
\(Y_t^*\leq a_t\) in every period. Restricting each Bellman minimization
to \(x\leq y\leq a_t\) preserves the full-information optimal value
for every \(t\in[T]\) and \(x\in[0,a_t]\).
\end{theorem}

\begin{proof}{Proof.}
Condition \eqref{eq:coverage-condition} implies that some \(y<a_t\)
satisfies \(F_t(y)>q_t\). Every myopic minimizer, including the largest
one \(S_t^m\), then lies below \(a_t\). Lemma
\ref{lemma:myopic base stock} gives \(S_t^*\leq S_t^m<a_t\).
If \(X_t^*\leq a_t\), then
\[
Y_t^*=\max\{X_t^*,S_t^*\}\leq a_t,\qquad
X_{t+1}^*=(Y_t^*-D_t)^+\leq a_t\leq a_{t+1}.
\]
Induction establishes the trajectory claim from any admissible tail state.
The restriction therefore excludes no such optimal trajectory and preserves
its value. \Halmos
\end{proof}

\paragraph{Why periodwise coverage alone does not ensure cap closure.}
Take two periods, unit holding and shortage costs, \(x_1=0\), and
\(\bar D=1\). Let \(D_1=3/4\) with probability \(3/4\) and zero
otherwise, while \(D_2=0\) surely. Set \(\lambda_1=1\) and
\(\lambda_2=1/4\). Both periodwise CDF values equal one, above
\(q_t=1/2\). Nevertheless, \(S_2^*=0\) and the first-period objective is
\[
U_1(y)=
\begin{cases}
9/16-y/4,&0\leq y\leq3/4,\\
2y-9/8,&3/4\leq y\leq1.
\end{cases}
\]
Its unique minimizer is \(S_1^*=3/4\). On the zero-demand event,
\(X_2^*=3/4>\lambda_2\): the later reported region is not closed
under optimal carryover. The carry-safe cap is instead \(a_1=a_2=1/4\),
and \(F_1^-(a_1)=1/4\) fails \eqref{eq:coverage-condition}, as it should.
This example concerns failure of cap closure, not impossibility of learning
this particular instance.

\subsection{Exact Truncation and Decision Equivalence}
\begin{proposition}[Exact truncation and policy-gap preservation]
\label{prop:exact-truncation}
Under the conditions of Theorem~\ref{theorem:decision coverage}, define
\(W_t=\min\{D_t,a_t\}\) and
\(\kappa_t=\sum_{s=t}^T p_s\mathbb E[D_s-W_s]\).
Write \(V_{t,D}^{\pi},V_{t,W}^{\pi}\) for lost-sales policy values under
the original and truncated demands, and use a superscript \(*\) for their
optimal values. For every \(t\in[T]\), \(x\in[0,a_t]\), and policy
\(\pi\) whose actions remain within the caps,
\[
V_{t,D}^{\pi}(x)=V_{t,W}^{\pi}(x)+\kappa_t,\qquad
V_{t,D}^{\pi}(x)-V_{t,D}^*(x)
=V_{t,W}^{\pi}(x)-V_{t,W}^*(x).
\]
The two systems have identical sales and inventory trajectories when coupled
with the same demands and policy randomization.
\end{proposition}
\begin{proof}{Proof.}
For every \(y\leq a_t\), pathwise
\[
(y-D_t)^+=(y-W_t)^+,\qquad
(D_t-y)^+-(W_t-y)^+=D_t-W_t.
\]
Also \(\min\{D_t,y\}=\min\{W_t,y\}\), so the policy receives the same
sales history and chooses the same actions in both systems. Writing
\(C_t^D(y)=\mathbb E[c_t(y,D_t)]\) and
\(C_t^W(y)=\mathbb E[c_t(y,W_t)]\), the expected stage costs satisfy
\begin{equation}
\label{eq:truncation-constant}
C_t^D(y)-C_t^W(y)=p_t\mathbb E[D_t-W_t],
\qquad y\leq a_t.
\end{equation}
Summing gives the policy-independent shift \(\kappa_t\).
Theorem~\ref{theorem:decision coverage} provides a cap-feasible original
optimum. Lemma~\ref{lemma:myopic base stock}, truncated support
\([0,a_t]\), and nondecreasing caps provide a cap-feasible truncated optimum.
Minimizing over these policies gives
\(V_{t,D}^*(x)=V_{t,W}^*(x)+\kappa_t\), proving the gap identity. \Halmos
\end{proof}

\subsection{Sharp Learning Rate under Prescribed Archives}
\label{section:prescribed-archive-rate}
In the prescribed-archive setting of Section~\ref{section:censored-policy-lower},
qualified logs recover the truncated demand samples needed by Product ERM.

\begin{theorem}[Policy learning under maintained coverage]
\label{theorem:coverage product erm}
{Suppose the conditions of Theorem \ref{theorem:decision coverage}
hold, and use \(\mathcal J_t,m_t,W_t^j\) from Section
\ref{section:usable-fraction}. Condition on the complete recorded boundary
array. For a positive integer \(n\), if \(m_t\geq n\) in every period,
choose any \(n\) indices from each \(\mathcal J_t\) using only the boundaries
and form the empirical marginals of \(W_t^j\). For fixed
\(x_1\in[0,a_1]\), let \(\widehat\pi^\times\) be Product ERM with
\(S_1\geq x_1\), equivalently empirical DP under these marginals.
There is a universal constant \(C>0\) such that, for any
\(0<\delta<1\),
\[
n\geq
C\frac{C_c^2\bar D^2T^2}{\epsilon^2}
\left(1+\log\frac{1}{\delta}\right)
\]
implies
\[
V_1^{\widehat\pi^\times}(x_1)-V_1^*(x_1)\leq\epsilon
\]
with conditional probability at least \(1-\delta\). In particular, if
\(N\) raw observations are available in every period and
\(m_t\geq rN\) for every \(t\), the sufficient raw total is
\[
TN=O\!\left(\frac{T}{r}
+\frac{C_c^2\bar D^2T^3}{r\epsilon^2}\left(1+\log\frac1\delta\right)\right).
\]
}
\end{theorem}

\begin{proof}{Proof.}
For \(j\in\mathcal J_t\),
\(W_t^j=\min\{Z_t^j,a_t\}=\min\{D_t^j,a_t\}\), so the selected
variables are fully observed samples from \(W_t\).
Conditional on the complete boundary array, boundary-only selection preserves
within-period IID sampling and cross-period independence. By Lemma
\ref{lemma:myopic base stock} and empirical support \([0,a_t]\), empirical
DP thresholds do not exceed \(a_t\). Theorem~\ref{theorem:decision coverage}
gives a cap-feasible true optimum. Apply the full-observation argument of
Section~\ref{section:product-transfer} to the truncated laws, whose
normalized losses have range length at most \(C_c\bar D\).
The common-recursion bound \eqref{eq:horizon-free-ingredient} and Lemma
\ref{lemma:product empirical concentration} give, with conditional
probability at least \(1-\delta\),
\[
V_{1,W}^{\widehat\pi^\times}(x_1)-V_{1,W}^*(x_1)
\leq\frac{2C_c\bar D T}{\sqrt n}
\left(C_0+\sqrt{\frac{\log(1/\delta)}2}\right).
\]
Proposition~\ref{prop:exact-truncation} transfers this gap to the original
system. The stated choice of \(n\) makes it at most \(\epsilon\). If
\(m_t\geq rN\), taking
\(N=\lceil n/r\rceil\) proves the raw-observation order.
\Halmos
\end{proof}

\begin{corollary}[Minimax policy-learning rate under strict coverage]
\label{cor:coverage minimax}
In the prescribed-archive setting of Section~\ref{section:censored-policy-lower},
for even \(T\geq4\), \(r\in(0,1]\), and \(0<\epsilon\leq T/512\),
\[
\mathcal M_{\rm cov}^*(T,\epsilon,r)
=\Theta\!\left(\frac{T^3}{r\epsilon^2}\right).
\]
\end{corollary}
\begin{proof}{Proof.}
Proposition~\ref{prop:rounded-coverage-lower} gives necessity and
Theorem~\ref{theorem:coverage product erm} gives sufficiency. Its rounding
term \(T/r\) is absorbed in the stated accuracy range. \Halmos
\end{proof}
The factor \(1/r\) counts all prescribed logs, including discarded ones.
The benchmark also controls truncated values, but the original-cost offset
\(\sum_{s=t}^T p_s\mathbb E[D_s-W_s]\) remains unknown, as illustrated by
Proposition~\ref{prop:censored-value-nonidentification}.

\subsection{Checking Coverage from the Same Archive}
\begin{proposition}[Observable coverage test]
\label{prop:coverage certificate}
{Fix \(0<\delta<1\) and caps as in Section~\ref{section:decision-coverage}; condition on the
complete boundary array. Use the qualified observations, with \(m_t\geq1\),
and let
\[
\mathcal J_t=\{j:b_t^j\geq a_t\},\quad m_t=|\mathcal J_t|,\quad
\widehat F_t^-(a_t)=\frac1{m_t}\sum_{j\in\mathcal J_t}
\mathbf{1}\{Z_t^j<a_t\}.
\]
Because \(b_t^j\geq a_t\), the displayed indicator equals
\(\mathbf{1}\{D_t^j<a_t\}\). Define
\[
r_t=\sqrt{\frac{\log(2T/\delta)}{2m_t}}.
\]
With probability at least \(1-\delta\), passing the test
\(\widehat F_t^-(a_t)-r_t>q_t\) in every period implies
\eqref{eq:coverage-condition}. Moreover, for margins \(\gamma_t>0\), if
\(F_t^-(a_t)\geq q_t+\gamma_t\) and
\(m_t>2\gamma_t^{-2}\log(2T/\delta)\), the test passes
simultaneously with probability at least \(1-\delta\).}
\end{proposition}

\begin{proof}{Proof.}
{Apply Hoeffding's inequality to each Bernoulli indicator and union
bound over periods, conditional on the boundaries. On the resulting event,
\(|\widehat F_t^-(a_t)-F_t^-(a_t)|\leq r_t\) for all \(t\). The first claim
is immediate. Under the margin condition, the sample-size requirement gives
\(r_t<\gamma_t/2\), so
\(\widehat F_t^-(a_t)-r_t\geq F_t^-(a_t)-2r_t>q_t\). \Halmos}
\end{proof}

The usable fraction counts logs reaching a cap, and the margin measures quantile
separation. This sufficient test is not necessary for identification.
Its observations can also fit the policy.

\begin{corollary}[Shared-data coverage checking and policy learning]
\label{cor:shared-data-learning}
Fix \(0<\delta<1\), \(x_1\in[0,a_1]\), and the caps and sampling
protocol of Theorem~\ref{theorem:coverage product erm}, without assuming
coverage. Apply Proposition~\ref{prop:coverage certificate}
to the same qualified observations, with failure budget \(\delta/2\),
and fit Product ERM using a boundary-selected subset satisfying that theorem's
count with learning budget \(\delta/2\). Return the fitted policy only if
the test passes. Then, conditional on the boundaries,
\[
\mathbb P\!\left(\text{test passes and }
V_1^{\widehat\pi^\times}(x_1)-V_1^*(x_1)>\epsilon\right)\leq\delta.
\]
If \(F_t^-(a_t)\geq q_t+\gamma_t\) with \(\gamma_t>0\) for all \(t\), a sufficiently
large universal constant \(C'\) makes the per-period requirements
\begin{equation}
\label{eq:shared-data-count}
m_t\geq C'\max\!\left\{
\frac{C_c^2\bar D^2T^2}{\epsilon^2}
\left(1+\log\frac2\delta\right),\quad
\gamma_t^{-2}\log\frac{4T}{\delta}\right\}
\end{equation}
sufficient for passing and returning an \(\epsilon\)-optimal policy
simultaneously with probability at least \(1-\delta\). The maximum reflects
sample reuse; it does not authorize choosing caps from the same demands.
For \(T\geq2\), \(0<\epsilon\leq T\), constant confidence and
\(C_c=\bar D=1\), usable fraction at least \(r\) gives a sufficient
raw total \(O(r^{-1}\sum_{t=1}^T
\max\{T^2\epsilon^{-2},\gamma_t^{-2}\log T\})\) with period-specific
counts.
\end{corollary}
\begin{proof}{Proof.}
Conditional on the boundaries, let \(E_{\rm cov}\) be the simultaneous
CDF event and \(E_{\rm fit}\) the truncated uniform-deviation event,
each with failure probability at most \(\delta/2\). Bounded independent
truncated demand ensures \(E_{\rm fit}\) without coverage, and its support
keeps empirical optimal actions below the nondecreasing caps. A union bound gives
\(\mathbb P(E_{\rm cov}\cap E_{\rm fit})\geq1-\delta\), without
event independence. On this intersection, passing implies strict coverage
and a cap-feasible true optimum, so truncation transfers the regret bound.
Passing with excess gap therefore occurs only off this intersection.
With positive margins, \eqref{eq:shared-data-count} ensures passing on
\(E_{\rm cov}\) by Proposition~\ref{prop:coverage certificate}.
Divide each period's usable count by \(r\) for the raw total, and the stated
accuracy range absorbs rounding. \Halmos
\end{proof}

\section{Concluding Remarks}
\label{section: conclusion}
We conclude that policy learning and valuation face different information limits for multistage stochastic inventory control. Inherited inventory can make valuation harder, and censoring can leave it unidentified even when carry-safe truncation attains the policy-learning limit. Demand-adaptive logging, positive lead times, and necessary multiperiod coverage conditions
are natural next steps for future studies.

\section*{Data and Code Accessibility}
The synthetic experiment source, fixed configurations, seeds, checks, and
saved figure inputs are stored in the supplementary files.
The source guide identifies the scripts and records used by the two
companion figures. The synthetic inputs are generated by the scripts. No proprietary data are needed for these illustrations.
\clearpage

\clearpage
\ECSwitch
\ECHead{E-Companion for ``Information Limits of Multistage Inventory Control: Learning, Valuation, and Censoring''}
\section{Supplementary Illustrations}
\label{section: simulation}
This section contains numerical illustrations of known-policy valuation
and boundary-qualified truncation. They are implementation and teaching
examples, not independent evidence for minimax rates or algorithmic
superiority. Scripts, fixed configurations, seeds, and figure inputs are
available in supplementary files.

\subsection{Valuation When the Optimal Policy Is Known}
For the stationary family of Theorem
\ref{theorem:stationary value lower bound}, take unit costs, initial stock
\(x_1=1\), and \(\rho=\mathbb P(D=0)=1-1/(2T)\), with demand one
otherwise. Supply the optimal threshold zero, so policy error is exactly
zero. From \(M\) pooled observations, let \(Z\) count zero demands and
estimate the policy value by \(v_T(Z/M)\), where
\[
v_T(u)=\sum_{k=0}^{T-1}\bigl[(1-u)+(2u-1)u^k\bigr].
\]
This formula evaluates the supplied policy even if \(Z/M<1/2\).
We use \(T\in\{10,20,40,80\}\),
\(M=\max\{2,\operatorname{round}(sT^3)\}\),
\(s\in\{.001,.004,.016,.064,.256\}\), and 20,000 replications per cell.

Figure~\ref{fig:valuation} shows nonzero value-estimation error despite the
known optimal decision rule. At larger budgets, RMSE approaches the local
sensitivity approximation \(v_T'(\rho)\sqrt{\rho(1-\rho)/M}\); small
budgets show finite-sample departures. The comparison isolates uncertainty
in the value of carried inventory, not an algorithm ranking. The horizontal
normalization by \(T^3\) is motivated by the theorem, not fitted to the data.
Theorem~\ref{theorem:stationary value lower bound}, not the nonzero error of
this particular estimator, establishes valuation difficulty. Supplying the
policy does not itself demonstrate the class-wide policy--value comparison.
Indeed, since \(Z\sim\operatorname{Bin}(M,\rho)\), the risk is exactly
\[
\mathbb E\bigl[(v_T(Z/M)-v_T(\rho))^2\bigr]
=\sum_{z=0}^M\binom Mz\rho^z(1-\rho)^{M-z}
\bigl[v_T(z/M)-v_T(\rho)\bigr]^2.
\]
The archived Monte Carlo curve approximates this finite sum.

\begin{figure}[htbp]
\centering
\includegraphics[width=\linewidth]{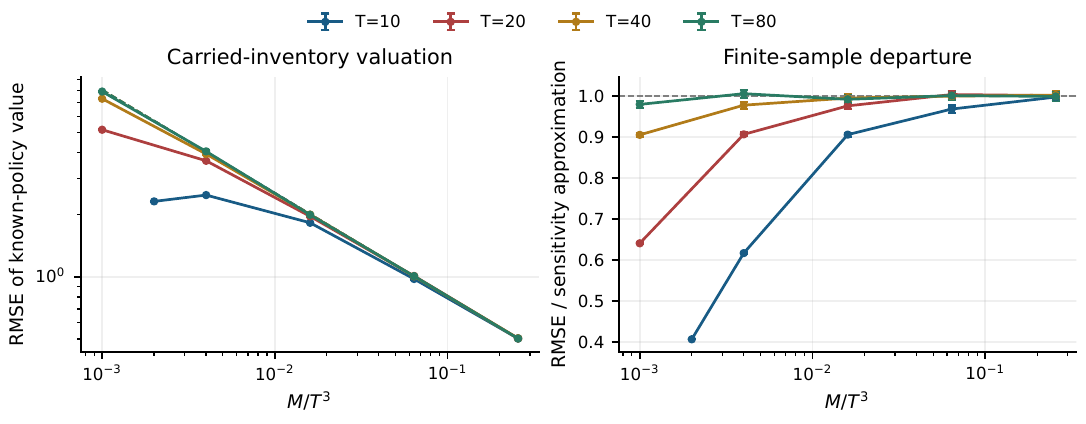}
\caption{Known-policy valuation. Left: RMSE and the sensitivity approximation
(dashed). Right: their ratio. Bars show approximate 95\% Monte Carlo intervals
for RMSE, obtained from the standard error of squared error and the delta
method; they are not confidence bounds for an individual value estimate.}
\label{fig:valuation}
\end{figure}

\subsection{Exact Truncation and Boundary-Qualified Learning}
Set \(T=20\), unit costs, \(x_1=0\), and cap \(a_t=7/8\).
Odd-period demand is \(1/4\) or \(1/2\), with upper-point probability
\(1/2+0.05\theta_i\) and alternating signs \(\theta_i\). Even-period
demand is \(3/4\) with probability \(3/4\), and one otherwise.
Every period satisfies strict decision coverage. For each raw per-period
budget \(B\in\{64,256,1024,4096\}\) and
\(r\in\{1,1/2,1/4,1/8\}\), the boundary schedule selects exactly
\(m=rB\) indices in every period by an independent random permutation,
assigns boundary \(7/8\) to those indices, and assigns \(1/8\) to the
remaining \((1-r)B\) indices. Selection is independent of the demand
values; the counts are fixed, not binomial. Each of 400 replications fits
Product ERM using the \(rB\) qualified observations. The boundary-blind
DP fits all sales as if they were fully observed demand. It is a
misspecification diagnostic, not a competitive censored-demand learner.

All reported percentages use the original-demand optimum:
\begin{equation}
\label{eq:ec-original-gap}
\operatorname{Gap}(\widehat\pi)
=100\,\frac{V_{1,D}^{\widehat\pi}(0)-V_{1,D}^*(0)}{V_{1,D}^*(0)}.
\end{equation}
Here \(V_{1,D}^*(0)=1.75\), whereas \(V_{1,W}^*(0)=1.4375\).
The finite demand and inventory lattice permits policy evaluation by
finite summation, without test Monte Carlo. In the 6,400 archived
replications, the \emph{absolute} gaps
\(V_{1,D}^{\widehat\pi}(0)-V_{1,D}^*(0)\) and
\(V_{1,W}^{\widehat\pi}(0)-V_{1,W}^*(0)\) agree within
\(1.2\times10^{-15}\). The omitted cost is the policy-independent
constant \(T/64=0.3125\), as predicted by
\eqref{eq:truncation-constant}. This identity preserves absolute gaps;
percentages agree only when the same denominator is used.

Figure~\ref{fig:truncation} distinguishes raw and qualified budgets.
For \(r=1/8\), the qualified learner's mean gap falls from about
6.5\% to 0.15\% across the displayed budgets. At equal qualified counts,
the learner receives statistically equivalent inputs, so alignment by
\(rB\) is built into the design. The universal \(1/r\) lower bound
comes from Theorem~\ref{theorem:coverage lower bound}, not this alignment.
Gap agreement checks the truncation identity in
Proposition~\ref{prop:exact-truncation}.

\paragraph{An exact explanation of the large plateau.}
Every low-boundary observation equals \(1/8\), since true demand is
always at least \(1/4\). Thus the boundary-blind empirical marginal is
\[
\widehat P_t^{\rm sales}=(1-r)\delta_{1/8}
+r\widehat P_t^{\rm qualified}.
\]
For \(r<1/2\), its unique myopic minimizer is \(1/8\). Ordering to
\(1/8\) in every period attains all empirical stagewise minima and
leaves zero inventory, so it is also dynamically optimal from zero,
independently of the qualified realizations. Under the true law, the ten
odd--even pairs give
\[
\begin{split}
V_{1,D}^*(0)&=10\left(\frac9{80}+\frac1{16}\right)=1.75,\\
V_{1,D}^{\pi_{1/8}}(0)&=10\left(\frac38-\frac18\right)
+10\left(\frac{13}{16}-\frac18\right)=9.375.
\end{split}
\]
The five positive and five negative odd-period signs average to odd-period
mean demand \(3/8\). The absolute excess cost is therefore \(7.625\),
and the exact relative gap is
\[
100\,\frac{9.375-1.75}{1.75}=435.7142857\ldots\%.
\]
The coincident \(r=1/4\) and \(r=1/8\) plateaus are predicted by
the construction, rather than a finite-sample phenomenon. A relative cost
gap can exceed 100\%; here the policy costs about 5.36 times the optimum.

\paragraph{Half coverage and optimizer selection.}
At \(r=1/2\), exactly half the observations equal \(1/8\), and the
empirical objectives admit ties. Exact smallest-minimizer selection would
again give the all-\(1/8\) policy and the same 435.7142857\% gap at
every budget. The archived implementation uses the \(1/40\) action
lattice and double-precision Bellman objectives. Its routine scans feasible actions in descending order and
replaces the incumbent when \(Q_t(y)\leq Q_t(y_{\rm best})\), with no
tie tolerance. It therefore selects the smallest action among exactly equal
\emph{computed} objectives; roundoff can distinguish mathematically tied
actions. The plotted \(r=1/2\) curve retains that numerical convention
and is not the exact smallest-minimizer curve. Its lower true cost is
sensitive to optimizer selection, not evidence of a different learning rate.

A separate controlled audit with 20 replications per cell
confirms the distinction:
for one half-coverage realization at \(B=64\), its actions from zero
include \(0.15\), \(0.175\), \(0.2\), and \(0.7\); an exact
rational Bellman calculation verifies that every returned action from zero
is optimal for that empirical problem, while the smallest such action is
\(1/8\) in every period. This audit is not a rerun of the archived
6,400 replications. At \(r=1\), the shared sales and qualified samples
have identical empirical marginals. The same deterministic solver therefore
returns identical policies and gaps replication by replication; this
identity also holds in the controlled audit.

\begin{figure}[htbp]
\centering
\includegraphics[width=\linewidth]{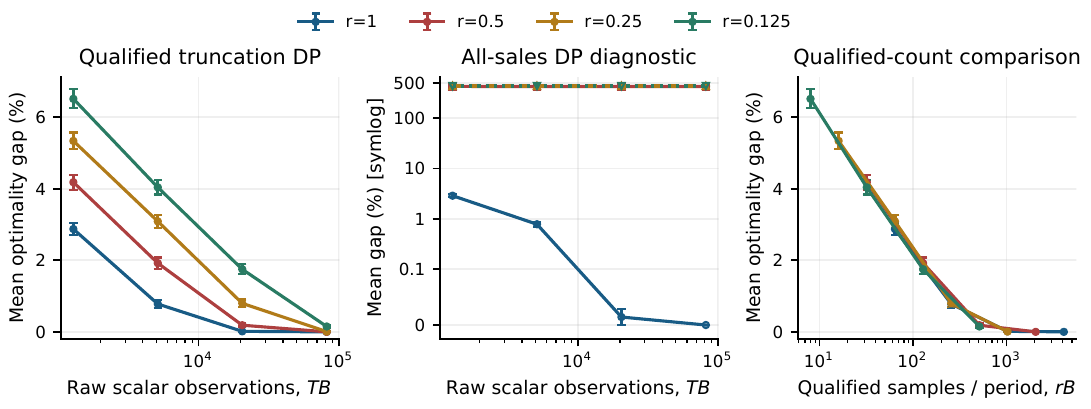}
\caption{Boundary-qualified Product ERM and boundary-blind DP
(middle panel: misspecification diagnostic, labeled ``All-sales DP diagnostic'').
The first two panels count all raw observations \(TB\); the third counts
qualified observations per period \(rB\). Gaps use the original optimum
\(1.75\) in \eqref{eq:ec-original-gap}. Bars are 95\% Student-$t$
intervals for the mean over 400 independent replications, not policy-level
guarantees. The middle panel uses a symmetric-log vertical scale, linear
within 0.1 percentage points of zero. The exact plateau for \(r<1/2\)
is 435.714\%; the \(r=1/4\) and \(r=1/8\) curves overlap. The
archived \(r=1/2\) curve uses floating-point comparisons without a tie
tolerance, as explained above. At \(r=1\), both methods coincide.}
\label{fig:truncation}
\end{figure}

\paragraph{An exact coverage-check calculation.}
Set \(T=20\), \(q_t=1/2\), cap \(7/8\), and
\(\mathbb P(D_t=1/4)=1/2+\gamma\), with remaining mass at one.
Apply Proposition~\ref{prop:coverage certificate} at \(\delta=0.05\)
to independent observations within and across periods. For \(m\) usable observations per
period, the exact simultaneous pass probability is
\[
\left[\mathbb P\!\left\{\operatorname{Bin}(m,1/2+\gamma)
>m\left(1/2+\sqrt{\frac{\log(2T/\delta)}{2m}}\right)\right\}\right]^T.
\]
No simulation is needed for this calculation. The proposition's sufficient
usable counts per period are
335, 1,337 and 5,348 for margins 0.2, 0.1 and 0.05. Halving the margin
requires about four times as many usable observations.
\end{document}